%% file: fokas-diagonalize-examples.tex
\documentclass[11pt,reqno,oneside,a4paper]{article}
\usepackage[a4paper,includeheadfoot,left=25mm,right=25mm,top=00mm,bottom=20mm,headheight=20mm]{geometry}
\input{texhead-main} 
\input{texhead-acknowledge} 
\input{texhead-project} 
\author{
    D. A. Smith \\
    \footnotesize Division of Science, Yale-NUS College, Singapore, \\ \footnotesize and Department of Mathematics, National University of Singapore, Singapore \\
    \footnotesize\href{mailto:dave.smith@yale-nus.edu.sg}{\texttt{dave.smith@yale-nus.edu.sg}}
}
\newcommand\theauthorshort{D. A. Smith}
\title{Fokas diagonalization}
\renewcommand{\runningtitle}{Fokas diagonalization}
\date{\today}

\begin{document}
\maketitle
\thispagestyle{fancy}

\begin{abstract}
    A method for solving linear initial boundary value problems was recently reimplemented as a true spectral transform method.
    As part of this reformulation, the precise sense in which the spectral transforms diagonalize the underlying spatial differential operator was elucidated.
    That work concentrated on two point initial boundary value problems and interface problems on networks of finite intervals.
    In the present work, we extend these results, by means of three examples, to new classes of problems: problems on semiinfinite domains, problems with nonlocal boundary conditions, and problems in which the partial differential equation features mixed derivatives.
    We show that the transform pair derived via the Fokas transform method features the same Fokas diagonalization property in each of these new settings, and we argue that this weak diagonalization property is precisely that needed to ensure success of a spectral transform method.
\end{abstract}




\section{Introduction} \label{sec:Introduction}

Fourier transform methods for solving initial boundary value problems for partial differential equations were instrumental to the advances of mathematical physics in the 19th century.
They essentially all work by reexpressing a spatial differential operator as a diagonal multiplication operator acting in the spectral domain.
This diagonalization reduces the spatiotemporal partial differential equation to an ordinary differential equation in the time variable only.
After determining the solution of the ordinary differential equation, an inverse transform is used to map back from the spectral domain to coordinate space, yielding the solution of the original initial boundary value problem.

In the present work, we describe an advance on this method.
Specifically, we explain how a family of solution methods for initial boundary value problems, known collectively as the Fokas transform method, or unified transform method, and developed over the past quarter century, can be interpreted as an advance on the classical Fourier transform method.
We identify, within the Fokas transform method, pairs of transforms and inverse transforms that can be used to solve each problem.
We show that, although these transforms do not diagonalize the relevant spatial differential operators in the usual sense, they each posses precisely the diagonalization property that is required for their use in a spectral transform method.
This weak diagonalization property is motivated in~\S\ref{ssec:Introduction.ClassicalTransform}--\ref{ssec:Introduction.GeneralTransformMethod}, and characterized informally in criterion~\ref{crit:FokasDiagonalization}.
In~\S\ref{sec:Results}, it is precisely stated and proved for three examples which are all beyond the class covered in the recent article~\cite{ABS2022a}.

\subsection{The classical spectral transform method} \label{ssec:Introduction.ClassicalTransform}

Suppose we wish to solve a problem like
\begin{prob*}[Half line Dirichlet problems for the heat equation]
    \begin{subequations} \label{eqn:IBVPClassical1}
    \begin{align}
        \label{eqn:IBVPClassical1.PDE} \tag{\theparentequation.PDE}
        [\partial_t-\partial_x^2]q(x,t) &= 0 & (x,t) &\in (0,\infty) \times (0,T), \\
        \label{eqn:IBVPClassical1.IC} \tag{\theparentequation.IC}
        q(x,0) &= Q(x) & x &\in[0,\infty), \\
        \label{eqn:IBVPClassical1.BC} \tag{\theparentequation.BC}
        q(0,t) &= 0 & t &\in[0,T],
    \end{align}
    in which $Q\in\mathcal{S}[0,\infty)$, the space of smooth functions on the half line, rapidly decaying along with all of their derivatives.
    \end{subequations}
    The differential operator $L$ defined by
    \begin{equation} \label{eqn:OperatorClassical1}
        L\phi=-\phi', \qquad \Dom L = \left\{ \phi\in \mathcal{S}[0,\infty) : \phi(0)=0 \right\}
    \end{equation}
    is such that, interpreting $L$ as always acting in the spatial variable, we may express~\eqref{eqn:IBVPClassical1.PDE} as $[\partial_t +L]q(x,t)=0$.
\end{prob*}
Consider some hypothetical transform $F$ which accepts a function of the spatial variable $x$ and returns a function of a new ``spectral'' variable $\la$.
We will apply this transform in the spatial variable to functions of space and timeto yield functions of $\la$ and $t$.
We suppose that this can be done in such a way that the transform commutes with the temporal derivative operator.
Suppose also that the transform is linear.

We can apply such a transform, to~\eqref{eqn:IBVPClassical1.PDE} to get
\[
    \frac{\D}{\D t} F[q](\la;t) + F[L q](\la;t) = 0.
\]
If $F$ has further the \emph{diagonalization} property that $F[L\phi](\la)=\la^2F[\phi](\la)$, then this simplifies to the temporal ordinary differential equation
\[
    \frac{\D}{\D t} F[q](\la;t) + \la^2 F[q](\la;t) = 0
\]
for $F[q](\la;t)$.
Its solution is
\[
    F[q](\la;t) = \re^{-\la^2t} C(\la),
\]
and evaluation at $t=0$ combined with appeal to~\eqref{eqn:IBVPClassical1.IC} establishes that $C(\la)=F[Q](\la)$.
Finally, we suppose that there exists another transform $F^{-1}$ which allows us to map back from functions of the spectral variable to functions of $x$, and that this second transform is an inverse of the original transform in the sense that, for all $x$, $F^{-1}[F[\phi]](x) = \phi(x)$.
Then
\[
    q(x,t) = F^{-1}\big[F[q](\argdot;t)\big](x) = F^{-1}\left[\re^{-\argdot^2t} F[Q] \right](x).
\]

This is a representation of the solution of problem~\eqref{eqn:IBVPClassical1} which relies only on the initial datum $Q$ and this speculative spectral transform.
So it remains only to find such a transform and the method is complete.
It is worth noting that, because we have not explicitly used~\eqref{eqn:IBVPClassical1.BC} in the solution method, they must be connected to the transform itself.
Consequently, if problem~\eqref{eqn:IBVPClassical1} and another with different boundary conditions are to have different solutions, then they must have different associated transforms.
However, this makes the method very general; by simply selecting a different transform, the method is immediately applicable to the Neumann or Robin problems for the heat equation on the half line.
The method also requires little modification before its application to partial differential equations with different $L$, or even different domains.

For this problem, of course, such $F$ is well known: the Fourier sine transform.
For the Neumann problem, the $F$ should be the Fourier cosine transform.
For the finite interval Dirichlet heat problem, the discrete Fourier sine transform, whose inverse is more commonly called the Fourier sine series, is the right transform.
But with a new transform required for each problem, we quickly exhaust the classical spectral transforms.

\subsection{Derivation of transforms} \label{ssec:Introduction.DerivationTransforms}
For finite interval problems, the appropriate transform is typically derived by separation of variables and solution of an appropriate Sturm-Liouville problem.
The classical Sturm-Liouville theory does not apply to problems of higher order, nonselfadjoint problems, or problems with more general types of boundary conditions, but much is known for many such problems.
Two point boundary value problems have seen the most attention, where Birkhoff's work~\cite{Bir1908b} was instrumental in showing that the crucial completeness and orthogonality results of Sturm-Liouville theory may be extended to a broad class.
This work was greatly expanded on over the next few decades, but Jackson and Hopkins soon identified examples where completeness of eigenfunctions fails~\cite{Jac1915a,Hop1919a}, dooming the usual approach to definition of transform pairs using eigenfunctions of the spatial differential operator and those of its adjoint.
Operators have been classified as ``regular'', ``irregular'', and ``degenerate'', with various definitions of each class, but the general theme of regular operators having all the properties we need to construct transform pairs, irregular operators having some impediments that make the construction more delicate, and degenerate operators, such as those identified by Jackson and Hopkins, usually considered beyond scope.
Full surveys are given by Locker~\cite{Loc2000a,Loc2008a}, with updates for irregular operators and beyond two point operators provided by Freiling~\cite{Fre2012a}.

An alternative approach is required to derive the appropriate transform pair, at least for degenerate irregular operators.
As the Fokas transform method is a spectral method, it is reasonable to hope that it may be a source for the appropriate transfoms.
As we shall argue, the situation is slightly more complex: the transform method itself must be modified to admit the transforms derived via the Fokas transform method, but the modification is natural.
A brief discussion of the problems solved via the Fokas transform method is appropriate.
Rather than attempting a full survey, we shall refer to a few examples that emphasize the classes of problems solved, because such variations may impact on the spectral transform method.
The Fokas transform method has been used to solve problems on the finite interval~\cite{FP2001a,Smi2012a}, the half line~\cite{FW2018a}, and interface domains~\cite{DPS2014a,DSS2016a,DS2020a}.
It has also been used to solve problems with multipoint~\cite{PS2018a} and nonlocal~\cite{MS2018a} conditions replacing the usual boundary conditions.
Extension to partial differential equations involving mixed derivatives have been addressed~\cite{DV2013a,CGM2021a}, and the method is also well understood for linear systems of partial differential equations~\cite{DGSV2018a}.
The Fokas transform method is much broader than the examples listed so far, having been applied to elliptic and hyperbolic equations without a time variable~\cite{Fok2001a,Cro2015a,CL2018a} and integrable semilinear equations.
These, along with semidiscrete problems~\cite{BH2008a}, and problems where the boundaries~\cite{FP2007a,FPX2019a} or boundary conditions~\cite{GPV2019a,ST2022a} move in time, are not discussed in this paper, because spectral methods in which the spectrum is time dependent are necessarily more complex.
See~\cite{Fok2008a,Pel2014a} and their references and citations for a picture of the broader Fokas transform method and~\cite{DTV2014a} for an introduction.

\subsection{Example problems} \label{ssec:Introduction.Problems}
We list here three problems which cannot be treated using the classical spectral transforms.
All of these problems have been solved using the Fokas transform method, and references are provided for the derivation of their corresponding Fokas transforms.
We select these particular problems for attention because their solution representations are relatively simple, yet the problems are sufficiently varied both to highlight the broad applicability of the Fokas transform method itself and to demonstrate that our reinterpretation of this method is equally universal.
Among these examples appear no two point boundary value problems, multipoint problems on finite intervals, nor interface problems on networks of finite intervals, because there already exists a complete charecterization of Fokas diagonalization in these settings~\cite{ABS2022a}.

\begin{prob*}[Half line Neumann problem for the Stokes equation]
    \begin{subequations} \label{eqn:IBVPStokes}
    \begin{align}
        \label{eqn:IBVPStokes.PDE} \tag{\theparentequation.PDE}
        [\partial_t+\partial_x^3]q(x,t) &= 0 & (x,t) &\in (0,\infty) \times (0,T), \\
        \label{eqn:IBVPStokes.IC} \tag{\theparentequation.IC}
        q(x,0) &= Q(x) & x &\in[0,\infty), \\
        \label{eqn:IBVPStokes.BC} \tag{\theparentequation.BC}
        q_x(0,t) &= 0 & t &\in[0,T],
    \end{align}
    in which $Q\in\mathcal{S}[0,\infty)$.
    \end{subequations}
    We also define the differential operator $L$ by
    \begin{equation} \label{eqn:OperatorStokes}
        L\phi=\phi''', \qquad \Dom L = \left\{ \phi\in \mathcal{S}[0,\infty) : \phi'(0)=0 \right\}
    \end{equation}
    to represent the spatial part of the problem.
\end{prob*}
The Dirichlet version of problem~\eqref{eqn:IBVPStokes} was solved using the Fokas transform method in~\cite[\S3.3]{DTV2014a}.
It is straightforward to adapt their argument to the Neumann case.
The finite interval two point analogue of $L$, with two supplementary boundary conditions provided at the other end of the spatial interval is precisely the operator Jackson and Hopkins identified as having incomplete eigenfunctions~\cite{Jac1915a,Hop1919a}.

\begin{prob*}[Finite interval problem for the heat equation with a nonlocal condition]
    \begin{subequations} \label{eqn:IBVPHeat}
    \begin{align}
        \label{eqn:IBVPHeat.PDE} \tag{\theparentequation.PDE}
        [\partial_t-\partial_x^2] q(x,t) &= 0 & (x,t) &\in (0,\infty) \times (0,T), \\
        \label{eqn:IBVPHeat.IC} \tag{\theparentequation.IC}
        q(x,0) &= Q(x) & x &\in[0,1], \\
        \label{eqn:IBVPHeat.BC} \tag{\theparentequation.BC}
        \int_0^1 K(y)q(y,t)\D y = 0 &= q_x(1,t) & t &\in[0,T],
    \end{align}
    in which $Q\in\operatorname{C}[0,1]$ and $K:[0,1]\to\RR$ is both supported in a neighbourhood of $0$ and globally sufficiently smooth.
    \end{subequations}
    The coresponding differential operator $L$ is given by
    \begin{equation} \label{eqn:OperatorHeat}
        L\phi=-\phi'', \qquad \Dom L = \left\{ \phi\in \operatorname{C}[0,1] : \int_0^1 K(y)\phi(y)\D y = 0 = \phi'(1) \right\}.
    \end{equation}
\end{prob*}
Problem~\eqref{eqn:IBVPHeat} was solved using the Fokas transform method in~\cite{MS2018a}.
It has a physical application in the problem of determining the concentration of a translucent mixture using a light sensor of finite width.

\begin{prob*}[Half line Dirichlet problem for the linearized BBM equation]
    \begin{subequations} \label{eqn:IBVPlBBM}
    \begin{align}
        \label{eqn:IBVPlBBM.PDE} \tag{\theparentequation.PDE}
        [\partial_t(1-\partial_x^2)+\partial_x]q(x,t) &= 0 & (x,t) &\in (0,\infty) \times (0,T), \\
        \label{eqn:IBVPlBBM.IC} \tag{\theparentequation.IC}
        q(x,0) &= Q(x) & x &\in[0,\infty), \\
        \label{eqn:IBVPlBBM.BC} \tag{\theparentequation.BC}
        q(0,t) &= 0 & t &\in[0,T],
    \end{align}
    \end{subequations}
    where $Q\in\mathcal{S}[0,\infty)$ and $Q(0)=0$.
    We also define the differential operators $L$ and $M$ by
    \begin{subequations} \label{eqn:OperatorlBBM}
    \begin{align} \label{eqn:OperatorlBBM.L}
        L\phi &= \phi',    & \Dom L &= \left\{ \phi\in \mathcal{S}[0,\infty) : \phi(0)=0 \right\}, \\
        M\phi &= 1-\phi'', & \Dom M &= \Dom L,
    \end{align}
    \end{subequations}
    so that the partial differential and boundary conditions can be represented as $[\partial_tM+L]q(\argdot,t)=0$.
\end{prob*}
Problem~\eqref{eqn:IBVPlBBM} was solved in~\cite{DV2013a} using the Fokas transform method, as part of the full class of Robin problems.
We specialise here to the Dirichlet problem so that the exposition may be aided by simpler formulae.
Problem~\eqref{eqn:IBVPlBBM} represents the small amplitude linearization of the bidirectional water wave model derived by Benjamin, Bona, and Mahoney~\cite{BBM1972a}.

\subsection{A more general spectral transform method} \label{ssec:Introduction.GeneralTransformMethod}
Suppose we aim to solve one of the above problems.
The classical Fourier sine and cosine transforms (or their discrete analogues for the finite interval problem) will not work.
Indeed, for the above problems, there are no known transforms that have both properties of diagonalizing the spatial differential operator and being invertible.
Therefore, our simple transform method will not succeed.
We provide here the archetype of a more general transform method which, as we argue in~\S\ref{sec:Results}, is applicable to these problems.
We decribe the method for problem~\eqref{eqn:IBVPStokes}, but the method is identical for problem~\eqref{eqn:IBVPHeat} and requires only natural generalization to problem~\eqref{eqn:IBVPlBBM}.

Suppose we have a transform $F$ and apply it in the spatial variable to~\eqref{eqn:IBVPStokes.PDE}, obtaining
\[
    \frac{\D}{\D t} F[q](\la;t) + F[Lq](\la;t) = 0,
\]
for a certain set of complex $\la$.
Because our transform may not diagonalize the differential operator $L$, we must admit the possibility that $L$ is diagonalized with a remainder:
\begin{equation} \label{eqn:GeneralTransform0}
    F[Lq](\la;t) = \omega(\la)F[q](\la;t) + R[q](\la;t).
\end{equation}
Then
\[
    \frac{\D}{\D t} F[q](\la;t) + \omega(\la)F[q](\la;t) + R[q](\la;t) = 0.
\]
Integration by parts tells us that the ``eigenvalue'' $\omega(\la)$ must be $-\ri\la^3$ and the boundary terms are collected into the remainder transform $R[q](\la;t)$.

We solve the ordinary differential equation for $F[q](\la;t)$, treating the remainder transform as if it were an inhomogeneity.
Indeed, multiplying by $\re^{\omega(\la)t}$ yields
\[
    \frac{\D}{\D t} \left( \re^{\omega(\la)t} F[q](\la;t) \right) + \re^{\omega(\la)t} R[q](\la;t),
\]
and integration in time from $0$ to $t$, followed by application of~\eqref{eqn:IBVPStokes.IC} and rearrangement yields
\begin{equation} \label{eqn:GeneralTransform1}
    F[q](\la;t) = \re^{-\omega(\la)t}F[Q](\la) - \int_0^t \re^{\omega(\la)(s-t)} R[q](\la;s).
\end{equation}

Arriving at equation~\eqref{eqn:GeneralTransform1} has required no special properties of the transform other than linearity and commutativity with the temporal derivative, both in the first step.
We have not yet assumed that the transform diagonalizes $L$.
At this point, so that we may obtain an expression for $q$ itself, we must assume that the transform is invertible.
We denote the inverse by $F^{-1}$, and assume it too is linear.
Then equation~\eqref{eqn:GeneralTransform1} simplifies to
\begin{equation} \label{eqn:GeneralTransform2}
    q(x,t) = F^{-1} \left[ \re^{-\omega t}F[Q] \right](x) - F^{-1} \left[ \int_0^t \re^{(s-t)\omega} R[q](\argdot;s) \D s\right](x).
\end{equation}
Unfortunately, our earlier pretence notwithstanding, $R[q]$ is not data, so equation~\eqref{eqn:GeneralTransform2} does not provide an effective representation of the solution to problem~\eqref{eqn:IBVPStokes}.
To proceed, let us suppose further that the latter term evaluates to zero whenever $q$ satisfies the boundary conditions.
Then the solution of the initial boundary value problem has been obtained:
\begin{equation} \label{eqn:GeneralTransform3}
    q(x,t)=F^{-1} \left[ \re^{-\omega t}F[Q] \right](x).
\end{equation}

The extra assumption is a restriction on $R$, so may be seen as part of the replacement for the diagonalization property of the transform.
Therefore, the requirements on the transform $F$ to make the spectral transform method work are:
\begin{crit} \label{crit:Invertibility}
    $F$ is invertible, and both $F$ and its inverse $F^{-1}$ are linear.
\end{crit}
\begin{crit} \label{crit:FokasDiagonalization}
    $F$ diagonalises the differential operator $L$ (which describes the spatial part of the initial boundary value problem) in the sense of equation~\eqref{eqn:GeneralTransform0}, with remainder $R$ having the property that, provided $q$ is sufficiently smooth and satisfies the boundary conditions, the latter term of equation~\eqref{eqn:GeneralTransform2} evaluates to $0$.
\end{crit}

Clearly, each initial boundary value problem will require its own transform $F$.
But transforms obeying analogues of both criteria~1 and~2 have been constructed for problems posed on the finite interval, with arbitrary constant coefficient differential operator and any linear boundary conditions~\cite{ABS2022a}.
Because these transform pairs were discovered via the Fokas transform method, we call the weak diagonalization criterion~2 \emph{Fokas diagonalization}.
As we shall argue below, there exist transforms that are tailored to each of the problems introduced above, which also exhibit Fokas diagonalization, so the above generalized transform method may be applied without hindrance.
Specifically, for each problem, we shall define the transform pair and prove criteria~\ref{crit:Invertibility} and~\ref{crit:FokasDiagonalization} as theorems.
Beyond~\cite{ABS2022a}, the papers~\cite{FS2016a,PS2016a,Smi2015a} provide an earlier view of Fokas diagonalization and the spectral transform method for two point problems and half line problems in which the spatial operator has monomial character.

\section{Results} \label{sec:Results}

\subsection{Half line Neumann problem for the Stokes equation} \label{ssec:Results.Stokes}

Adapting~\cite[\S3.3]{DTV2014a}, the Fokas transform pair
\begin{subequations} \label{eqn:defnF.Stokes}
\begin{align}
    F[\phi](\la) &=
    \begin{cases}
        \int_0^\infty \re^{-\ri\la y} \phi(y) \D y & \la \in \RR, \\
        \int_0^\infty \phi(y) \left[ \alpha^2 \re^{-\ri\alpha\la y} + \alpha \re^{-\ri\alpha^2\la y} \right] \D y \qquad & \la \in \partial D^+,
    \end{cases} \\
    F^{-1}[f](x) &= \frac1{2\pi} \int_{\RR\cup \partial D^+} \re^{\ri\la x} f(\la) \D\la \hspace{9.245em} x \in [0,\infty),
\end{align}
\end{subequations}
may be derived.
Here, $D^+$ is the sector $\arg(\la)\in(\frac\pi3,\frac{2\pi}3)$, $\partial D^+$ is the positively oriented contour traversing its boundary, the contour denoted $\RR$ is oriented in the increasing sense, and the primitive cube root of unity $\alpha=\re^{2\pi\ri/3}$.

\begin{prop} \label{prop:Stokes.Validity}
    If $F,F^{-1}$ are defined by equations~\eqref{eqn:defnF.Stokes}, then they are both linear and, for all $\phi$ sufficiently smooth and all $x\in(0,1)$, $F^{-1}[F[\phi]](x) = \phi(x)$.
\end{prop}

\begin{proof}
    These are integral transforms, so they inherit linearity from the improper definite integrals and contour integrals of which they are composed.
    
    By definition,
    \begin{align} \notag
        F^{-1}[F[\phi]](x)
        &= \frac1{2\pi} \int_{-\infty}^\infty \re^{\ri\la x} F[\phi](\la) \D\la + \frac1{2\pi} \int_{\partial D^+} \re^{\ri\la x} F[\phi](\la) \D\la \\
        \label{eqn:Stokes.Validity.Proof1}
        &= \phi(x) + \frac1{2\pi} \int_{\partial D^+} \re^{\ri\la x} F[\phi](\la) \D\la,
    \end{align}
    where the second equality holds for all $x\in(0,\infty)$ at which $\phi$ is continuous, and is justified by the usual Fourier inversion theorem for piecewise smooth $\phi$.
    It remains only to show that the remaining contour integral term evaluates to $0$.
    
    The definition of $F[\phi](\la)$ is analytically extensible from $\partial D^+$ to a neighbourhood of the sector $D^+$.
    Doing so will necessarily overwrite the definition of $F[\phi](\la)$ on part of $\RR$, but that is inconsequential as we have already removed that part of the domain from consideration; we, for the purposes of the rest of this proof, consider $F[\phi](\la)$ as being defined by its $\partial D^+$ formula everywhere on $\CC$.
    Integrating by parts, we see that, as $\la\to\infty$ from within $\clos D^+$,
    \[
        F[\phi](\la) = \frac\ri{\la} \left\{ \left[ \phi(y)\left( \alpha\re^{-\ri\alpha\la y} + \alpha^2\re^{-\ri\alpha^2\la y} \right) \right]_0^\infty
        - \int_0^\infty \phi'(y)\left( \alpha\re^{-\ri\alpha\la y} + \alpha^2\re^{-\ri\alpha^2\la y} \right) \D y \right\} = \lindecayla,
    \]
    and this decay is uniform in $\arg(\la)$ within the given sector.
    Hence, by Jordan's lemma and Cauchy's theorem, the remaining integral on the right of equation~\eqref{eqn:Stokes.Validity.Proof1} evaluates to $0$.
\end{proof}

\begin{thm} \label{thm:Stokes.Diag}
    Suppose $F,F^{-1}$ are defined by equations~\eqref{eqn:defnF.Stokes}.
    There exists a remainder transform $R$ for which, for all $\la\in \RR\cup \partial D^+$ and all $\phi\in\Dom L$,
    \begin{equation} \label{eqn:Stokes.Diag.1}
        F[L\phi](\la) = -\ri\la^3 F[\phi](\la) + R[\phi](\la).
    \end{equation}
    Moreover, if, for all $t\in[0,T]$ $q(\argdot,t)\in\Dom L$ then, for all $t\in[0,T]$ and all $x\in(0,\infty)$,
    \begin{equation} \label{eqn:Stokes.Diag.2}
        F^{-1}\left[ \int_0^t \re^{-\ri\argdot^3(s-t)} R[q](\argdot;s) \D s \right](x) = 0.
    \end{equation}
\end{thm}

\begin{proof}
    Integrating by parts thrice in the definition of $F[\phi](\la)$ and applying the boundary condition $\phi'(0)=0$, we find that equation~\eqref{eqn:Stokes.Diag.1} holds with
    \begin{equation*}
        R[\phi](\la) =
        \begin{cases}
            -r(\la;\phi) & \mbox{if } \la\in\RR, \\
             r(\la;\phi) & \mbox{if } \la\in\partial D^+,
        \end{cases} \qquad\qquad
        r(\la;\phi) = \phi''(0) - \la^2\phi(0),
    \end{equation*}
    in which we think of the polynomial $r(\argdot;\phi)$ as having domain all of $\CC$.
    Note that it is entire, and is $o(\la^3)$ as $\la\to\infty$.
    
    Let $E^+$ be the union of sectors $\arg(\la)\in(0,\frac\pi3)\cup(\frac{2\pi}3,\pi)$.
    Then, integrating by parts, as $\la\to\infty$ from within $\clos E^+$,
    \begin{align*}
        \int_0^t \re^{-\ri\la^3(s-t)} r(\la;q(\argdot,s)) \D s
        &= \frac{\ri}{\la^3} \left\{ \left[ \re^{-\ri\la^3(s-t)} r(\la;q(\argdot,s)) \right]_0^t + \int_0^t \re^{-\ri\la^3(s-t)} r(\la;q_t(\argdot,s)) \D s \right\} \\
        &= \lindecayla,
    \end{align*}
    uniformly in $\arg(\la)$ within those closed sectors.
    Therefore, by Jordan's lemma and Cauchy's theorem,
    \[
        \int_{\partial E^+} \re^{\ri\la x} \int_0^t \re^{-\ri\la^3(s-t)} r(\la;q(\argdot,s)) \D s \D\la = 0.
    \]
    Hence, by comparing the sector boundaries $\partial E^+$ and $\partial D^+$,
    \begin{align*}
        \int_{\partial D^+} \re^{\ri\la x} \int_0^t \re^{-\ri\la^3(s-t)} r(\la;q(\argdot,s)) \D s \D\la
        &= \int_{\partial D^+ \cup \partial E^+} \re^{\ri\la x} \int_0^t \re^{-\ri\la^3(s-t)} r(\la;q(\argdot,s)) \D s \D\la \\
        &= \int_{-\infty}^\infty \re^{\ri\la x} \int_0^t \re^{-\ri\la^3(s-t)} r(\la;q(\argdot,s)) \D s \D\la.
    \end{align*}
    Therefore,
    \begin{multline*}
        F^{-1}\left[ \int_0^t \re^{-\ri\argdot^3(s-t)} R[q](\argdot;s) \D s \right](x) \\
        = \frac1{2\pi} \int_0^\infty \re^{\ri\la x} \int_0^t \re^{-\ri\la^3(s-t)} \Big( r(\la;q(\argdot,s)) - r(\la;q(\argdot,s)) \Big) \D s \D\la
        = 0. \qedhere
    \end{multline*}
\end{proof}

In proposition~\ref{prop:Stokes.Validity} we have fulfilled criterion~\ref{crit:Invertibility}, and theorem~\ref{thm:Stokes.Diag} establishes criterion~\ref{crit:FokasDiagonalization}.
Therefore, the general transform method of~\S\ref{ssec:Introduction.GeneralTransformMethod} can be implemented for this problem to derive solution~\eqref{eqn:GeneralTransform3}.

The argument is presented above in such a way that it requires minimal modification for the Dirichlet problem; the transform pair may be found in~\cite[\S3.3]{DTV2014a} and $r(\la;\phi)$ is replaced with $\phi''(0)+\ri\la\phi'(0)$, which is still entire and $o(\la^3)$.

\subsection{Finite interval problem for the heat equation with a nonlocal condition} \label{ssec:Results.Heat}

As derived in~\cite{MS2018a}, the Fokas transform pair is
\begin{subequations} \label{eqn:defnF.Heat}
\begin{align}
    F[\phi](\la) &=
    \begin{cases}
        \displaystyle\int_0^1 \re^{-\ri\la y} \phi(y) \D y & \la \in \RR, \\
        -\zeta^+(\la;\phi)/\Delta(\la) & \la \in \partial D_\rho^+, \\
        -\re^{-\ri\la}\zeta^-(\la;\phi)/\Delta(\la) \qquad\qquad\hspace{0.42em} & \la \in \partial D_\rho^-,
    \end{cases} \\
    F^{-1}[f](x) &= \frac1{2\pi} \int_{\RR\cup D_\rho^+\cup D_\rho^-} \re^{\ri\la x} f(\la) \D\la \qquad\qquad x \in [0,1],
\end{align}
\end{subequations}
where
\begin{align}
    \Delta(\la) &= \int_0^1 K(y) \cos([1-y]\la) \D y, \\
    \label{eqn:defnzetap.Heat}
    \zeta^+(\la;\phi) &= \int_0^1 K(y) \cos([1-y]\la) \int_0^y \re^{-\ri\la z} \phi(z) \D z \D y + \int_0^1 K(y) \re^{-\ri\la y} \int_y^1 \cos([1-z]\la) \phi(z) \D z \D y, \\
    \label{eqn:defnzetam.Heat}
    \zeta^-(\la;\phi) &= \int_0^1 K(y) \int_y^1 \sin([z-y]\la) \phi(z) \D z \D y,
\end{align}
and $D^\pm_\rho=\{\la\in\CC^\pm$ such that $\Re(\la^2)<0$ and $\abs\la>\rho\}$ for $\rho$ sufficiently large to ensure all zeros of $\Delta$ have imaginary part bounded between $\pm \rho/\sqrt2$, and $\partial D^\pm_\rho$ indicates the positively oriented contour traversing the boundary of the region $D^\pm_\rho$.
The existence of such $\rho$ was proved in~\cite[lemma~2.1]{MS2018a}.

\begin{prop} \label{prop:Heat.Validity}
    If $F,F^{-1}$ are defined by equations~\eqref{eqn:defnF.Heat}, then they are both linear and, for all $\phi$ sufficiently smooth and all $x\in(0,1)$, $F^{-1}[F[\phi]](x) = \phi(x)$.
\end{prop}

\begin{proof}
    Linearity follows from linearity of the definite real integral and the contour integral.
    With $\zeta^\pm$ defined by equations~\eqref{eqn:defnzetap.Heat}--\eqref{eqn:defnzetam.Heat}, it was shown in~\cite[lemma~2.2]{MS2018a} that,
    provided $K$ is of bounded variation, $K$ is continuous and supported in a neighbourhood of $0$, and $\lVert\phi'\rVert_\infty$ is bounded,
    then, as $\la\to\infty$ from within $\clos D_\rho^\pm$, $\zeta^\pm(\la;\phi)/\Delta(\la) = \lindecayla$, uniformly in $\arg(\la)$.
    The ratios $\zeta^\pm/\Delta$ are, because of the choice of $\rho$ sufficiently large, analytic in neighbourhoods of $\clos D_\rho^\pm$.
    It follows by Jordan's lemma and Cauchy's theorem that the integrals along the boundaries of $D_\rho^\pm$ appearing in $F^{-1}[F[\phi]](x)$ both evaluate to $0$.
    Therefore,
    \[
        F^{-1}[F[\phi]](x) = \frac1{2\pi} \int_{-\infty}^\infty \re^{\ri\la x} F[\phi](\la) \D\la.
    \]
    The proposition follows from the usual Fourier inversion theorem.
\end{proof}

\begin{thm} \label{thm:Heat.Diag}
    Suppose $F,F^{-1}$ are defined by equations~\eqref{eqn:defnF.Heat}.
    There exists a remainder transform $R$ for which, for all $\la\in \RR\cup D_\rho^+\cup D_\rho^-$ and all $\phi\in\Dom L$,
    \begin{equation} \label{eqn:Heat.Diag.1}
        F[L\phi](\la) = \la^2 F[\phi](\la) + R[\phi](\la).
    \end{equation}
    Moreover, if $q:[0,1]\times[0,T] \to \CC$ is such that, for all $t\in[0,T]$, $q(\argdot,t)\in\Dom L$, and $q$ is sufficiently smooth in $t$, then, for all $t\in[0,T]$ and all $x\in(0,1)$,
    \begin{equation} \label{eqn:Heat.Diag.2}
        F^{-1}\left[ \int_0^t \re^{\argdot^2(s-t)} R[q](\argdot;s) \D s \right](x) = 0.
    \end{equation}
\end{thm}

\begin{proof}
    Integration by parts and application of the boundary and nonlocal conditions yield that
    \[
        R[\phi](\la) =
        \begin{cases}
            r_-(\la;\phi)\re^{-\ri\la}-r_+(\la;\phi) & \mbox{if } \la\in\RR, \\
            r_+(\la;\phi) & \mbox{if } \la\in\partial D_\rho^+, \\
            r_-(\la;\phi) & \mbox{if } \la\in\partial D_\rho^-,
        \end{cases}
    \]
    where
    \[
        r_+(\la;\phi) := -\phi'(0) - \ri\la\phi(0),
        \qquad
        r_-(\la;\phi) := -\ri\la\phi(1).
    \]
    Note that, for each of the three contours on which $R[\phi]$ is defined, it may be analytically extended to all of $\CC$, yielding a triply defined function on $\CC$, with each definition entire.
    
    We denote by $E^\pm_\rho$ the sets $\{\la\in\CC^\pm$ such that $\Re(\la^2)>0$ and $\abs\la>\rho\}$.
    Then, integrating by parts,
    \[
        \int_0^t \re^{\la^2(s-t)} r_+(\la;q(\argdot,s)) \D s = \bigoh{\la^{-2}}
    \]
    as $\la\to\infty$ from within $\clos E_\rho^\pm$, uniformly in $\arg(\la)$.
    Hence, by Jordan's lemma,
    \[
        \int_{\partial E_\rho^+} \re^{\ri\la x} \int_0^t \re^{\la^2(s-t)} r_+(\la;q(\argdot,s)) \D s \D\la = 0.
    \]
    It follows, by comparing the paths of the contours, that
    \[
        \int_{\partial D_\rho^+} \re^{\ri\la x} \int_0^t \re^{\la^2(s-t)} r_+(\la;q(\argdot,s)) \D s \D\la = \int_{\gamma_\rho^+} \re^{\ri\la x} \int_0^t \re^{\la^2(s-t)} r_+(\la;q(\argdot,s)) \D s \D\la,
    \]
    in which $\gamma_\rho^+$ is the contour that extents from $-\infty$ to $-\rho$, then follows the semicircular path in $\CC^+$ from $-\rho$ to $\rho$, then proceeds from $\rho$ to $\infty$.
    Because the integrand is entire, the contour $\gamma_\rho^+$ may be deformed to the real line.
    Similarly, but noting that the real part of $\la$ decreases as $\la$ traverses $\partial D_\rho^-$, whereas it increased when $\la$ followed $\partial D_\rho^+$,
    \[
        \int_{\partial D_\rho^-} \re^{\ri\la(x-1)}\int_0^t \re^{\la^2(s-t)} r_-(\la;q(\argdot,s)) \D s \D\la = -\int_{-\infty}^\infty \re^{\ri\la(x-1)} \int_0^t \re^{\la^2(s-t)} r_-(\la;q(\argdot,s)) \D s \D\la.
    \]
    
    Suppresing dependence of $r_\pm$ on $\la,q(\argdot;s)$, and using the previous two displayed equations to justify the second equality, it follows that
    \begin{align*}
        F^{-1}\left[ \int_0^t \re^{\argdot^2(s-t)} R[q](\argdot;s) \D s \right](x) \hspace{-10em} &\hspace{10em}
        = \int_{-\infty}^\infty \re^{\ri\la x} \int_0^t \re^{\la^2(s-t)} [r_-\re^{-\ri\la} - r_+] \D s \D\la \\
        &\hspace{3em} + \int_{\partial D_\rho^+} \re^{\ri\la x} \int_0^t \re^{\la^2(s-t)} r_+ \D s \D\la + \int_{\partial D_\rho^-} \re^{\ri\la (x-1)} \int_0^t \re^{\la^2(s-t)} r_- \D s \D\la \\
        &= \int_{-\infty}^\infty \re^{\ri\la x} \int_0^t \re^{\la^2(s-t)} [r_-\re^{-\ri\la} - r_+] \D s \D\la + \int_{-\infty}^\infty \re^{\ri\la x} \int_0^t \re^{\la^2(s-t)} r_+ \D s \D\la \\
        &\hspace{3em} - \int_{-\infty}^\infty \re^{\ri\la (x-1)} \int_0^t \re^{\la^2(s-t)} r_- \D s \D\la = 0.
        \qedhere
    \end{align*}
\end{proof}

Because proposition~\ref{prop:Heat.Validity} and theorem~\ref{thm:Heat.Diag} follow exactly the archetypes of criteria~\ref{crit:Invertibility} and~\ref{crit:FokasDiagonalization}, the transform method is effective at finding the solution of problem~\eqref{eqn:IBVPHeat}, with $\omega(\la)=\la^2$.

\subsection{Half line Dirichlet problem for the linearized BBM equation} \label{ssec:Results.lBBM}

Deconinck and Vasan show in~\cite[\S3]{DV2013a} that the Fokas transform pair is
\begin{subequations} \label{eqn:defnF.lBBM}
\begin{align}
    F[\phi](\la) &=
    \begin{cases}
        \displaystyle \int_0^\infty \re^{-\ri\la y} \phi(y) \D y & \la \in \RR, \\
        \displaystyle \frac{-1}{\la^2} \int_0^\infty \re^{\frac{-\ri}{\la} y} \phi(y) \D y \hspace{1.9em} & \la \in \partial \mathcal{C},
    \end{cases} \\
    F^{-1}[f](x) &= \frac1{2\pi} \int_{\RR\cup \mathcal{C}} \re^{\ri\la x} f(\la) \D\la \qquad\qquad x \in [0,\infty),
\end{align}
\end{subequations}
where $\mathcal{C}$ is a small positively oriented simple closed contour enclosing $\la=\ri$.

\begin{prop} \label{prop:lBBM.Validity}
    If $F,F^{-1}$ are defined by equations~\eqref{eqn:defnF.lBBM}, then they are both linear and, for all $\phi$ sufficiently smooth and all $x\in(0,\infty)$, $F^{-1}[F[\phi]](x) = \phi(x)$.
\end{prop}

\begin{proof}
    These are integral transforms, so they are linear.
    By definition,
    \[
        F^{-1}[F[\phi]](x)
        = \frac1{2\pi} \int_{-\infty}^\infty \re^{\ri\la x} \int_0^\infty \re^{-\ri\la y} \phi(y) \D y \D\la
        - \frac1{2\pi}\int_{\mathcal{C}} \frac1{\la^2} \int_0^\infty \re^{\frac{-\ri}{\la} y} \phi(y) \D y \D\la.
    \]
    The integral of the second contour integral is analytic in $\CC^+$, in which $\mathcal{C}$ lies.
    Therefore, by Cauchy's theorem, the second contour integral evaluates to $0$.
    The proposition follows by the usual Fourier inversion theorem.
\end{proof}

\begin{thm} \label{thm:lBBM.Diag}
    Suppose $F,F^{-1}$ are defined by equations~\eqref{eqn:defnF.lBBM}.
    There exists a remainder transform $R$ for which, for all $\la\in \RR\cup \mathcal{C}$ and all $\phi\in\Dom L = \Dom M$,
    \begin{subequations} \label{eqn:lBBM.Diag.1}
    \begin{align} \label{eqn:lBBM.Diag.1L}
        F[L\phi](\la) &= \omega_L(\la) F[\phi](\la) + R_L[\phi](\la), \\ \label{eqn:lBBM.Diag.1M}
        F[M\phi](\la) &= \omega_M(\la) F[\phi](\la) + R_M[\phi](\la),
    \end{align}
    \end{subequations}
    where
    \[
        \omega_L(\la) = \begin{cases} \ri\la & \mbox{if } \la\in\RR, \\ \frac\ri\la & \mbox{if } \la\in\mathcal{C}, \end{cases}
        \qquad\qquad
        \omega_M(\la) = \begin{cases} 1+\la^2 & \mbox{if } \la\in\RR, \\ 1+ \frac1{\la^2} & \mbox{if } \la\in\mathcal{C}. \end{cases}
    \]
    Moreover, if $q$ is sufficiently smooth, $q(\argdot,t)\in\Dom L=\Dom M$ then, for all $t\in[0,T]$ and all $x\in(0,\infty)$,
    \begin{equation} \label{eqn:lBBM.Diag.2}
        F^{-1}\left[ \int_0^t \re^{(s-t)\omega} \frac{R_L[q](\argdot;s)+R_M[q_t](\argdot;s)}{\omega_M} \D s \right](x) = 0,
    \end{equation}
    where
    $\omega(\la)=\omega_L(\la)/\omega_M(\la)=\ri\la/(1+\la^2)$.
\end{thm}

\begin{proof}
    Integration by parts and application of the boundary condition demonstrates equations~\eqref{eqn:lBBM.Diag.1} with $R_L=0$ and
    \[
        R_M[\phi](\la) =
        \begin{cases}
            \phi'(0) & \mbox{if } \la\in\RR, \\
            -\phi'(0)\frac1{\la^2} & \mbox{if } \la\in\mathcal{C}.
        \end{cases}
    \]
    Therefore, the left side of equation~\eqref{eqn:lBBM.Diag.2} is
    \begin{multline} \label{eqn:lBBM.Diag.Proof1}
        \frac1{2\pi} \int_{-\infty}^\infty \re^{\ri\la x} \int_0^t \re^{\frac{\ri\la}{1+\la^2}(s-t)} q_{xt}(0;s) \D s \frac1{1+\la^2} \D\la \\
        - \frac1{2\pi} \int_{\mathcal{C}} \re^{\ri\la x} \int_0^t \re^{\frac{\ri\la}{1+\la^2}(s-t)} q_{xt}(0;s) \D s \frac1{1+\frac1{\la^2}}\left(\frac1{\la^2}\right) \D\la.
    \end{multline}
    The integrand of the first contour integral in expression~\eqref{eqn:lBBM.Diag.Proof1} is analytic on $\CC\setminus\{\pm\ri\}$,
    It is straightforward to show that $\Re(\omega(\la))\geq0$ on $\clos\CC^+$ except on the closed disc with radius $1$ and center at $0$.
    Hence, as $\la\to\infty$ from within $\clos\CC^+$,
    \[
        \int_0^t \re^{\frac{\ri\la}{1+\la^2}(s-t)} q_{xt}(0;s) \D s \frac1{1+\la^2} = \bigoh{\la^{-2}},
    \]
    uniformly in $\arg(\la)$, and the same function is analytic everywhere but $\pm\ri$.
    Hence, by Jordan's lemma and Cauchy's theorem, the first contour integral in expression~\eqref{eqn:lBBM.Diag.Proof1} may be deformed from $\RR$ to $\mathcal{C}$, whereupon it cancels with the other contour integral in that expression.
    Thereby, equation~\eqref{eqn:lBBM.Diag.2} is established.
\end{proof}

\subsubsection*{Transform method}
Theorem~\ref{thm:lBBM.Diag} does not follow the archetype of criterion~\ref{crit:FokasDiagonalization}.
But~\eqref{eqn:IBVPlBBM.PDE} has form different from~\eqref{eqn:IBVPStokes.PDE}, for which the general transform method of~\S\ref{ssec:Introduction.GeneralTransformMethod} was developed, so it is unsurprising that the form of Fokas diagonalization is different.
Below, we implement the transform method applicable to problem~\eqref{eqn:IBVPlBBM}, noting how theorem~\ref{thm:lBBM.Diag} is used.

Applying the transform $F$ to~\eqref{eqn:IBVPlBBM.PDE}, we obtain
\[
    0 = \frac{\D}{\D t} F[Mq](\la;t) + F[Lq](\la;t),
\]
for all $\la\in\RR\cup\mathcal{C}$.
It follows from equations~\eqref{eqn:lBBM.Diag.1} that
\[
    0 = \frac{\D}{\D t} \Big( \omega_M(\la)F[q](\la;t) + R_M[q](\la;t) \Big) + \omega_L(\la)F[q](\la;t) + R_L[q](\la;t).
\]
Rearranging and multiplying by $\re^{\omega(\la)t}/\omega_M(\la)$, we find
\[
    0 = \frac{\D}{\D t} \left( \re^{\omega(\la)t} F[q](\la;t) \right) + \re^{\omega(\la)t} \frac{R_L[q](\la;t) + R_M[q_t](\la;t)}{\omega_M(\la)}.
\]
Integrating in time from $0$ to $t$, applying~\eqref{eqn:IBVPlBBM.IC} and rearranging, we find
\[
    F[q](\la;t) = \re^{-\omega(\la)t} F[Q](\la) - \int_0^t \re^{\omega(\la)(s-t)} \frac{R_L[q](\la;s)+R_M[q_t](\la;s)}{\omega_M(\la)} \D s.
\]
When we apply the inverse transform, because $q$ obeys~\eqref{eqn:IBVPlBBM.BC}, theorem~\ref{thm:lBBM.Diag} guarantees that the latter term evaluates to $0$.
Hence the solution of problem~\eqref{eqn:IBVPlBBM} is
\[
    q(x,t) = F^{-1}\left[ \re^{-\omega t} F[Q](\argdot;t) \right](x).
\]

\section{Conclusion} \label{sec:Conclusion}
We aimed both to show that the Fokas transform method can be reformulated to appear as a generalized spectral transform method for initial boundary value problems and to provide a unified characterization of Fokas diagonalization.

With the detailed work~\cite{ABS2022a} already demonstrating the former claim for finite interval problems with two point, multipoint and interface type boundary conditions, the ambit of the present work is to demonstrate the generalization to other settings by means of a few examples.
There exist important classes of initial boundary value problems for which the Fokas transform method has been implemented that are not represented among the examples studied above.
But significant progress has been made towards this aim.

Regarding the second aim, Fokas diagonalization in the first two examples exactly matches criterion~\ref{crit:FokasDiagonalization}, which also matches Fokas diagonalization for finite interval problems~\cite{ABS2022a}.
The final example demonstrates that the statement of Fokas diagonalization is properly considered as a property of the boundary value problem under study, rather than any particular (local) differential operator.
But it also reinforces the point that the statement of Fokas diagonalization may be read directly from the spectral transform method; it is precisely the necessary and sufficient condition for the spectral transform method to succeed.
Although, for the sake of brevity, not demonstrated in the current work, this extends to Fokas diagonalization for system problems such as those studied in~\cite{DGSV2018a,CGM2021a}.

\subsection{On pedagogy}
Having taught a few cohorts of undergraduates a typical course on boundary value problems, the author has observed that those learning spectral transform methods for the first time find it helpful to have an alternative viewpoint, or even introductory experience, in which the following two concepts are presented as separate:
\begin{enumerate}
    \item[C1.]{
        How a transform pair is used in a spectral method to solve an initial boundary value problem.
    }
    \item[C2.]{
        How the transform pair suitable for any particular problem may be derived.
    }
\end{enumerate}
It is to the detriment of students' learning, when the particular discrete Fourier transform appropriate for a given initial boundary value problem is derived by separation of variables only as part of a long solution method, and the definitions of the forward and inverse transforms are not explicitly identified as such.
Students may see solving a particular Sturm Liouville problem as part of solving a boundary value problem, but less commonly understand that they are deriving the spectral transform that will diagonalize the spatial differential operator for that problem.

Because C2 is quite difficult and technical, the relative simplicity and astonishingly broad applicability of C1 remain mysterious to many students.
The author has had some success in teaching first C1, with a hypothetical transform, and gently guiding students to discover for themselves C2: separation of variables and enough Sturm-Liouville theory to explictly construct the desired transform.

In a similar way, the Fokas transform method is usually presented in the literature as a monolithic method, at the end of which a solution has been derived, but the transform itself is rarely emphasized.
It is the author's intention that concept C1 for the Fokas transform method be essentially captured in~\S\ref{ssec:Introduction.GeneralTransformMethod}.
It is the author's hope that, by drawing attention to the crucial properties of the Fokas transform, criteria~\ref{crit:Invertibility} and~\ref{crit:FokasDiagonalization}, this work may lighten the burden of learning the Fokas transform method, and inspire more to study it.

\section*{Acknowledgement}

\AckNIDispHyd{The author}
\AckYNCSeed{The author also}

\bibliographystyle{amsplain}
{\small\bibliography{dbrefs}}

\end{document}

%% file: texhead-main.tex
\usepackage{amssymb,amsmath,amsthm}
\usepackage{xcolor,graphicx}
\usepackage{verbatim}
\usepackage{mathtools}
\usepackage{hyperref}
\usepackage{titling}
\usepackage{fancyhdr}
\newcommand{\runningtitle}{Running Title}
\newtheorem{thm}{Theorem}

\newtheorem{prop}[thm]{Proposition}

\theoremstyle{definition}

\newtheorem{crit}[thm]{Criterion}

\newtheorem*{prob*}{Problem}
\theoremstyle{remark}

\newcommand{\RR}{\mathbb R}              
\newcommand{\CC}{\mathbb C}              
\renewcommand{\Re}{\operatorname*{Re}} 
\newcommand{\D}{\ensuremath{\,\mathrm{d}}}
\newcommand{\ri}{\ensuremath{\mathrm{i}}}
\newcommand{\re}{\ensuremath{\mathrm{e}}}
\newcommand{\la}{\ensuremath{\lambda}}
\renewcommand{\epsilon}{\varepsilon}
\renewcommand{\geq}{\geqslant}

\providecommand{\clos}{\operatorname{clos}}
\newcommand{\abs}[1]{\left\lvert#1\right\rvert}

\usepackage{scalerel}
\usepackage{stackengine}
\stackMath
\newcommand\reallywidecheck[1]{%
\savestack{\tmpbox}{\stretchto{%
  \scaleto{%
    \scalerel*[\widthof{\ensuremath{#1}}]{\kern-.6pt\bigwedge\kern-.6pt}%
    {\rule[-\textheight/2]{1ex}{\textheight}}
  }{\textheight}%
}{0.5ex}}%
\stackon[1pt]{#1}{\scalebox{-1}{\tmpbox}}%
}

\newcommand\reallywidehat[1]{%
\savestack{\tmpbox}{\stretchto{%
  \scaleto{%
    \scalerel*[\widthof{\ensuremath{#1}}]{\kern-.6pt\bigwedge\kern-.6pt}%
    {\rule[-\textheight/2]{1ex}{\textheight}}
  }{\textheight}%
}{0.5ex}}%
\stackon[1pt]{#1}{\tmpbox}%
}

%% file: texhead-acknowledge.tex
\newcommand{\AckYNCSeed}[1]{#1 gratefully acknowledges support from Yale-NUS College seed grant IG21-SG101.}

\newcommand{\AckNIDispHyd}[1]{#1 would like to thank the Isaac Newton Institute for Mathematical Sciences for support and hospitality during programme \emph{Dispersive hydrodynamics: mathematics, simulation and experiments, with applications in nonlinear waves}, when work on this paper was undertaken. This work was supported by EPSRC Grant Number EP/R014604/1.}

%% file: texhead-project.tex
\usepackage{esint}

\usepackage{multirow}

\numberwithin{equation}{section}

\providecommand{\Dom}{\operatorname{Dom}}

\providecommand{\bigoh}[1]{\mathcal{O}\left(#1\right)}
\providecommand{\lindecayla}{\bigoh{\la^{-1}}}

\providecommand{\argdot}{{}\cdot{}}

\newsavebox{\accentbox}

\usepackage{mathrsfs} 

\definecolor{colorGAMMAcutsDARK}{rgb}{1,0.65,0}
\colorlet{colorGAMMAcuts}{colorGAMMAcutsDARK!25}
\definecolor{colorGAMMAaP}{rgb}{1,0.9,0.9}
\definecolor{colorGAMMAaM}{rgb}{0.9,1,0.9}
\definecolor{colorGAMMAzP}{rgb}{0.9,0.9,1}
\colorlet{colorGAMMAzM}{black!05}
\usepackage[most]{tcolorbox}
\newtcbox{\boxGAMMAc}[1][]{enhanced, colback=colorGAMMAcuts, frame style={opacity=0}, interior style={opacity=1}, nobeforeafter, tcbox raise base, shrink tight, extrude by=1mm, #1}
\newtcbox{\boxGAMMAaP}[1][]{enhanced, colback=colorGAMMAaP, frame style={opacity=0}, interior style={opacity=1}, nobeforeafter, tcbox raise base, shrink tight, extrude by=1mm, #1}
\newtcbox{\boxGAMMAaM}[1][]{enhanced, colback=colorGAMMAaM, frame style={opacity=0}, interior style={opacity=1}, nobeforeafter, tcbox raise base, shrink tight, extrude by=1mm, #1}
\newtcbox{\boxGAMMAzP}[1][]{enhanced, colback=colorGAMMAzP, frame style={opacity=0}, interior style={opacity=1}, nobeforeafter, tcbox raise base, shrink tight, extrude by=1mm, #1}
\newtcbox{\boxGAMMAzM}[1][]{enhanced, colback=colorGAMMAzM, frame style={opacity=0}, interior style={opacity=1}, nobeforeafter, tcbox raise base, shrink tight, extrude by=1mm, #1}